\documentclass[twoside,a4paper,12pt,centertags]{amsart}
\usepackage{amsmath,amssymb,verbatim,vmargin}
\usepackage[bookmarks=true]{hyperref}   % Hyperref in DVI and PDF

\theoremstyle{plain}
\newtheorem{thm}{Theorem}[section]
\newtheorem{lem}[thm]{Lemma}
\newtheorem{cor}[thm]{Corollary}
\newtheorem{prop}[thm]{Proposition}

\theoremstyle{definition}

\newtheorem{rem}[thm]{Remark}

\mathchardef\semic="303B
\newcommand{\dirac}{{\mathbf D}}

\newcommand{\R}{{\mathbb R}}
\newcommand{\N}{{\mathbb N}}
\newcommand{\C}{{\mathbb C}}

\newcommand{\mH}{{\mathcal H}}

\newcommand{\mZ}{{\mathcal Z}}
\newcommand{\mL}{{\mathcal L}}

\newcommand{\mA}{{\mathcal A}}

\newcommand{\mR}{{\mathcal R}}
\newcommand{\mT}{{\mathcal T}}
\newcommand{\mS}{{\mathcal S}}

\DeclareMathOperator{\re}{Re}

\newcommand{\nul}{\textsf{N}}
\newcommand{\ran}{\textsf{R}}
\newcommand{\dom}{\textsf{D}}

\newcommand{\clos}[1]{\overline{#1}}
\newcommand{\conj}[1]{\overline{#1}}

\newcommand{\sgn}{\text{{\rm sgn}}}
\newcommand{\barint}{\mbox{$ave \int$}}
\newcommand{\divv}{{\text{{\rm div}}}}
\newcommand{\curl}{{\text{{\rm curl}}}}

\newcommand{\tdd}[2]{\tfrac{\partial #1}{\partial #2}}

\newcommand{\wt}{\widetilde}

\newcommand{\ta}{{\scriptscriptstyle \parallel}}
\newcommand{\no}{{\scriptscriptstyle\perp}}
\newcommand{\pd}{\partial}
\newcommand{\ep}{\varepsilon}

\newcommand{\uT}{{\underline T}}
\newcommand{\bet}{{\mathcal B}}

\newcommand{\loc}{\text{{\rm loc}}}

\def\barint_#1{\mathchoice
            {\mathop{\vrule width 6pt
height 3 pt depth -2.5pt
                    \kern -8.8pt
\intop \kern -4pt}\nolimits_{#1}}%
            {\mathop{\vrule width 5pt height
3 pt depth -2.6pt
                    \kern -6.5pt
\intop \kern -4pt}\nolimits_{#1}}%
            {\mathop{\vrule width 5pt height
3 pt depth -2.6pt
                    \kern -6pt
\intop \kern -4pt}\nolimits_{#1}}%
            {\mathop{\vrule width 5pt height
3 pt depth -2.6pt
          \kern -6pt \intop \kern -4pt}\nolimits_{#1}}}
          
          \def\bariint_#1{\mathchoice
            {\mathop{\vrule width 10pt
height 3 pt depth -2.5pt
                    \kern -12.8pt
\intop \kern -10pt\intop \kern -4pt}\nolimits_{#1}}%
            {\mathop{\vrule width 9pt height
3 pt depth -2.6pt
                    \kern -10.5pt
\intop \kern -10pt\intop \kern -4pt}\nolimits_{#1}}%
            {\mathop{\vrule width 9pt height
3 pt depth -2.6pt
                    \kern -10pt
\intop \kern -10pt\intop \kern -4pt}\nolimits_{#1}}%
            {\mathop{\vrule width 9pt height
3 pt depth -2.6pt
          \kern -10pt \intop \kern -10pt\intop \kern -4pt}
      \nolimits_{  #1}}}

\usepackage{color}

\definecolor{gr}{rgb}   {0.,   0.8,   0. } 
\definecolor{bl}{rgb}   {0.,   0.5,   1. } 
\definecolor{mg}{rgb}   {0.7,  0.,    0.7}

\makeatletter
\@namedef{subjclassname@2010}{
  \textup{2010} Mathematics Subject Classification}
\makeatother

\title[block triangular coefficients]{On $L^2$ Solvability of BVPs for elliptic systems}

\author[P. Auscher]{Pascal Auscher}
\address{Univ. Paris-Sud, laboratoire de Math\'ematiques, UMR 8628 du CNRS, F-91405 {\sc Orsay}} 
\email{pascal.auscher@math.u-psud.fr}

\author[A. McIntosh]{Alan McIntosh}
\address{Centre for Mathematics and its 
Applications, Mathematical Sciences Institute,
Australian National University, {\sc Canberra} ACT 0200, 
Australia}
\email{Alan.McIntosh@maths.anu.edu.au}
\author[M. Mourgoglou]{Mihalis Mourgoglou}
\address{Univ. Paris-Sud, laboratoire de Math\'ematiques, UMR 8628 du CNRS, F-91405 {\sc Orsay}} 
\email{mihalis.mourgoglou@math.u-psud.fr}

\begin{document}

\begin{abstract}
In this article we prove solvability results for $L^2$ boundary value problems of  some elliptic systems $Lu=0$ on the upper half-space $\R^{n+1}_{+}, n\ge 1$, with transversally independent coefficients. We use the  first order  formalism introduced by Auscher-Axelsson-McIntosh and further developed with a better understanding of the classes of solutions in the subsequent work of Auscher-Axelsson. The interesting fact is  that we prove only half of the Rellich boundary inequality without knowing the other half.  

\end{abstract}


%\subjclass[2010]{}
%\keywords{}

\maketitle

\section{Introduction}

The  goal of this paper is to study Rellich type estimates for some elliptic systems by using the first order formalism introduced in \cite{AAM}. 

We begin by  recalling the classical Lax-Milgram existence  theorems of variational (or energy) solutions for divergence form second order partial differential equations and make clear how we understand them in unbounded domains. It is well-known that both Neumann and regularity problems have unique solutions in the energy class.  We reformulate this using the $DB$ formalism in order to represent these solutions via a semi-group. We then obtain factorisations of  the Dirichlet to Neumann map and of its inverse in $\dot H^{-1/2}$ topologies which are related to (abstract) boundary layer potentials.  This part is completely general and true for any such elliptic system.

In the last part, we assume that the coefficients are block triangular (see below for definition). Although such an assumption is rather restrictive, it brings up a  phenomenon which we think interesting on its own. Using these factorisations, we prove that,  in the  triangular situation which corresponds to making the conormal derivative proportional to the transversal derivative,  the Neumann to Dirichlet map
is bounded in the $L^2$ topology, that is, solutions (we shall make clear the meaning of solutions) satisfy half of the  boundary Rellich estimate 
$$
\|\nabla_{tan}u\|_{2}\le C\|\pd_{A}u\|_{2}
$$
which implies that the Neumann problem with $L^2$ data is solvable.  
For the adjoint situation, it is the Dirichlet to Neumann map that is bounded on $L^2$, hence the 
opposite Rellich inequality  for solutions
$$
\|\pd_{A}u\|_{2} \le C \|\nabla_{tan}u\|_{2}
$$
holds so that the regularity problem with $L^2$ data is solvable. In each case, we do not know about boundedness of the inverse, that is, we do not know  the other half of the Rellich estimate (and we do not expect it by any means) which is the alluded to phenomenon: usually Rellich estimates are proved both ways. We also show that  the Dirichlet problem with  $L^2$ data is solvable in the same block triangular situation as for the Neumann problem. 

We remark that the conjunction of the two situations is the block-diagonal case for which 
$$
\|\nabla_{tan}u\|_{2}\approx \|\pd_{A}u\|_{2}
$$
 is known to be equivalent to  the Kato square root estimate \cite{AHLMT}. 

As the reader might notice, a feature of our proof is the use of one equivalent formulation of the Kato square root estimate (for an auxiliary operator) which shows a tight connection between Rellich estimates and $Tb$ technology.

In a subsequent paper of two of us (the first and third named author) with S. Hofmann,  Neumann solvability results for Hardy space data are obtained in the triangular situations of this paper in the case of systems  of such pde's  satisfying the De Giorgi regularity condition.  

\bigskip

\paragraph{\bf Acknowledgments} {Auscher thanks the Mathematical Science Institute of the Australian National University for hospitality and support where this work started. McIntosh acknowledges support from the Australian Government through the Australian Research Council. Mourgoglou was supported by the Fondation Math\'emati\-que Jacques Hadamard and thanks the University Paris-Sud for hospitality. We also thank Steve Hofmann and Andreas Ros\'en for discussions pertaining to this work.}
 
\section{notation}

The system of  equations  is 
\begin{equation}  \label{eq:divform}
 (Lu)^\alpha(t,x)=  \sum_{i,j=0}^n\sum_{\beta= 1}^m \pd_i\Big( A_{i,j}^{\alpha, \beta}(x) \pd_j u^{\beta}(t,x)\Big) =0,\qquad \alpha=1,\ldots, m
\end{equation}
in $\R^{1+n}_+=(0,\infty)\times \R^n$,
where $\pd_0= \tdd{}{t}$ and $\pd_i= \tdd{}{x_{i}}$, if $i=1,\ldots,n$.
We assume
\begin{equation}   \label{eq:boundedmatrix}
  A=(A_{i,j}^{\alpha,\beta}(x))_{i,j=0,\ldots, n}^{\alpha,\beta= 1,\ldots,m}\in L_\infty(\R^n;\mL(\C^{2m})),
\end{equation}
and that $A$ is strictly accretive on  the subspace $\mH^0$  of $ L_2(\R^n;\C^{(1+n)m})$ defined by $(f_{j}^\alpha)_{j=1,\ldots,n}$ is curl free in $\R^n$ for all $\alpha$,  that is
there exists $\lambda>0$ such that for all $f\in\mH^0$
\begin{equation}   \label{eq:accrassumption}
  \sum_{i,j=0}^n\sum_{\alpha,\beta=1}^m \int_{\R^n} \re (A_{i,j}^{\alpha,\beta}(x)f_j^\beta(x) \conj{f_i^\alpha(x)}) dx\ge \lambda 
   \sum_{i=0}^n\sum_{\alpha=1}^m \int_{\R^n} |f_i^\alpha(x)|^2dx.
\end{equation}

The system \eqref{eq:divform} is always considered in the sense of distributions with weak solutions, that is,  $H^1_{loc}$ solutions in the particular domain under consideration.

We stress that our results are valid for any $m$ but set for notational simplicity $m=1$ from now on.  In this case, the accretivity condition above is equivalent to the usual pointwise accretivity 
condition 
$$
\re \sum_{i,j=0}^n A_{i,j}(x)\xi_j \conj{\xi_i} \ge \lambda 
   \sum_{i=0}^n |\xi_{i}|^2. $$
   It is convenient to write $A$ in the block form
$$
  A(x)= \begin{bmatrix} a(x) & b(x)\\ c(x) &d(x) \end{bmatrix}
$$
where $a$ is scalar-valued, $b,c$ vector-valued and $d$ $n\times n$ matrix-valued. 
Call $\mA$ the set of $2 \times 2$ block matrices $A$ with these properties.

All our estimates in this article depend only on ellipticity constants $\Lambda=\|A\|_{\infty}$ and the largest $\lambda$ in the accretivity inequality   \eqref{eq:accrassumption}.  For non negative quantities $a,b$, the notation $a\lesssim b$ means $a\le C b$ where $C$ is a constant depending on the parameters at hand. The notation $a \approx b$ means both $a\lesssim b$ and $b\lesssim a$.

\section{Energy solutions}

In this section we recall the usual construction of variational or energy solutions. Although this is fairly classical and mostly  follows \cite{AAM}, we need to stress a few points and set up some notation.
This uses the homogeneous Sobolev space $\dot H^1(\R^{1+n}_+)$ of $u\in L^2_{loc}(\R^{1+n}_+)$ such that  $\nabla_{t,x} u \in L^2(\R^{1+n}_+)$, equipped with the semi-norm
$\|u\|^2_{\dot H^1}:= \iint_{\R^{1+n}_+}|\nabla_{t,x} u|^2\, dtdx$ (note that we could have supposed $u\in \mathcal{D}'(\R^{1+n}_{+})$ as it is not hard to check that $u$ can be identified with a $L^2_{loc}(\R^{1+n}_+)$ function when $\nabla u \in L^2$). This is a Banach space modulo constants.

\begin{lem} \label{lem:density} $C^\infty_{0}(\overline{\R^{1+n}_+})$ is dense in $\dot H^1(\R^{1+n}_+)$. The trace on $\R^n$ is  bounded from $\dot H^1(\R^{1+n}_+)$ onto the homogoneous Sobolev space $\dot H^{1/2}(\R^n)$, which we  define as the closure of $C^\infty_{0}(\R^n)$ for the semi-norm $\left(\int_{\R^n}|\xi||\hat f(\xi)|^2\, d\xi\right)^{1/2}$ (with $\hat f$ designating the Fourier transform). 
\end{lem}

\begin{proof} Once the density is shown, the trace result is classical. Let $u\in \dot H^1(\R^{1+n}_+)$. As for any  compact subset $K$  of $\R^n$ and  $0<t_{0}<t_{1}<\infty$,
$$\int_{K}|u(t_{1}, x) - u(t_{0},x)|^2\, dx\le  (t_{1}-t_{0}) \iint_{\R^{1+n}_+} |\pd_{t}u(t,x)|^2\, dtdx,
$$
 $u$ extends to an element  $u\in L^2_{loc}(\overline{\R^{1+n}_+})$. By the reflection principle, we can extend $u$ to an element in  $  \dot H^1(\R^{1+n}) $ which we still call $u$.  We claim we can find a sequence $u_{k}$ of $L^2(\R^{1+n})$  functions with compact support that converges to $u$ in $ \dot H^1(\R^{1+n}_+)$. Indeed, by Mazur's lemma, it suffices to have a sequence with  weak convergence, that is such that $\nabla u_{k} \rightarrow  \nabla u$ weakly in $L^2$.  It is easy to see using Poincar\'e inequality, and this is where we need to know \textit{a priori} that $u\in L^2_{loc}$,  that $u_{k}= (u-c_{k})\varphi_{k}$ has this property when $c_{k}$ is the mean of $u$ on the ball $B(0,2^{k+1})$ and $\varphi_{k}(x)= \varphi(2^{-k}x)$ with $\varphi$ is a smooth function that vanishes outside $B(0,2)$ and that is 1 on $B(0,1)$. Finally, it suffices to mollify this sequence to conclude.  
\end{proof}

\begin{thm}\label{thm:Neu}  Given $\ell \in \dot H^{-1/2}(\R^n)$, there exists $u \in \dot H^1(\R^{1+n}_+)$, unique up to a constant, such that $\iint_{\R^{1+n}_+} A \nabla u \cdot \overline{\nabla \phi}\, dtdx = \langle \ell, \varphi\rangle$, for any $\phi \in \dot H^1(\R^{1+n}_+)$ with trace $\varphi$.  This function $u$ is a weak solution of $Lu=0$ in $\R^{1+n}_+$. We define the conormal derivative of $u$ at the boundary  to be $\pd_{\nu_{A}}u_{|t=0}=-\ell$ (we use the inward unit normal convention).
\end{thm}

We say that $u$ is the \emph{energy solution of the Neumann problem for $Lu=0$ with Neumann data} $-\ell$.

\begin{proof} This is just the Lax-Milgram theorem in the Hilbert space $\dot H^1(\R^{1+n}_+)/\C$ using that $\dot H^{-1/2}(\R^n)$ is the dual space of $\dot H^{1/2}(\R^n)$.  \end{proof}

\begin{thm} Given $f \in \dot H^{1/2}(\R^n)$, there exists $v \in \dot H^1(\R^{1+n}_+)$, unique up to a constant, such that $Lv=0$ in $\R^{1+n}_+$  and $v_{|t=0}=f$ with equality in $\dot H^{1/2}(\R^n)$. Furthermore, there exists $\ell \in \dot H^{-1/2}(\R^n)$ such that  $\iint_{\R^{1+n}_+} A \nabla u \cdot \overline{\nabla \phi}\, dtdx = \langle \ell, \varphi\rangle$ for any  $\varphi \in \dot H^{1/2}(\R^n)$ and any extension $\phi \in \dot H^1(\R^{1+n}_+)$ of $\varphi$. \end{thm}

We say that $v$ is the \emph{energy solution for the regularity problem $Lv=0$  with data $\nabla_{x}f$} and we have $\pd_{\nu_{A}}v_{|t=0}=-\ell$.  

\begin{proof} Pick an extension $w \in \dot H^1(\R^{1+n}_+) $ of $f$. Let $\dot H^1_{0}(\R^{1+n}_+) $ be the subspace of $ \dot H^1(\R^{1+n}_+) $ consisting of all $u$ with constant trace on $\R^n$ (alternately this is the closure of $C^\infty_{0}(\R^{1+n}_+)$ in $\dot H^1(\R^{1+n}_+) $). By the Lax-Milgram theorem, there exists a unique $u \in \dot H^1_{0}(\R^{1+n}_+) $ solving $$\iint_{\R^{1+n}_+} A \nabla u \cdot \overline{\nabla \phi}\, dtdx=- \iint_{\R^{1+n}_+} A \nabla w \cdot \overline{\nabla \phi}\, dtdx$$  for all  $\phi \in \dot H^1_{0}(\R^{1+n}_+)$. Then $v=u+w$ is the solution. 

Next, the integral $\iint_{\R^{1+n}_+} A \nabla u \cdot \overline{\nabla \phi}\, dtdx$ depends only on the trace modulo constants of $\phi \in \dot H^1(\R^{1+n}_+)$. Thus, the map $\varphi \mapsto \iint_{\R^{1+n}_+} A \nabla u \cdot \overline{\nabla \phi}\, dtdx$ is bounded from $\dot H^{1/2}(\R^n)$ to $\C$ and this defines $\ell$. 
\end{proof}

Observe that $\nabla_{x} $ is injective with closed range from $\dot H^{1/2}(\R^n)$  into $\dot H^{-1/2}(\R^n; \C^n)$. We set $\dot \mH^{-1/2}_{\ta}$ the range of this map and for a reason that will become clear later  also set $\dot \mH^{-1/2}_{\no}= \dot H^{-1/2}(\R^n)$. By a Fourier transform argument, one sees that $\dot \mH^{-1/2}_{\ta}= \mR (\dot\mH^{-1/2}_{\no})$, where $\mR=\nabla (-\Delta)^{-1/2}$ is the array of Riesz transforms on $\R^n$ (the Hilbert transform if $n=1$) and $\Delta$ is the ordinary self-adjoint Laplace operator on $L^2(\R^n)$. With this notation, one  defines the Neumann to Dirichlet map 
\begin{equation}
\label{eq:ND}
\Gamma^A_{ND}\ell= \nabla_{x}u_{|t=0}, \quad \ell \in  \dot\mH^{-1/2}_{\no}
\end{equation}
where  $u$ is the energy solution of the Neumann problem for $Lu=0$ with Neumann data $-\ell$ and the Dirichlet to Neumann map 
\begin{equation}
\label{eq:DN}
\Gamma^A_{DN}g= \pd_{\nu_{A}}v_{|t=0}, \quad  g \in  \dot\mH^{-1/2}_{\ta},
\end{equation} 
where $v$ is the energy solution of $Lv=0$ of the regularity problem with data $g$.

\begin{thm}\label{thm:inverse}
$\Gamma^A_{ND}$ is a bounded and invertible map from $\dot \mH^{-1/2}_{\no}$ onto $\dot\mH^{-1/2}_{\ta}$ with inverse $\Gamma^A_{DN}$.
\end{thm}

The proof is an obvious consequence of the above results with our definitions. We have 
$\Gamma^A_{ND}(\pd_{\nu_{A}}u_{|t=0})= \nabla_{x}u_{|t=0}$ and $\Gamma^A_{DN}(\nabla_{x}u_{|t=0})=\pd_{\nu_{A}}u_{|t=0}$ for any energy solution $u$ in the upper half-space of $Lu=0$.

We finish this section with a standard result.

\begin{lem} \label{lem:uniqueness}
Let $u\in \dot H^1(\R^{1+n})$ be a solution of $Lu=0$ in $\R^{1+n}$. Then $u=0$ (modulo constants). 
\end{lem}

\begin{proof} By definition, we have 
$$\iint_{\R^{1+n}} A \nabla u \cdot \overline{\nabla \phi}\, dtdx = 0$$
for all $\phi \in C_{0}^\infty(\R^{1+n})$, hence  for all $\phi\in \dot H^1(\R^{1+n})$ by density as in Lemma \ref{lem:density}. We conclude taking $\phi=u$ and using the accretivity of $A$.
\end{proof}

\section{The first order formalism}

Following \cite{AAM} and \cite{AA1}, we can characterize weak solutions  $u$ to the divergence form equation (\ref{eq:divform}), 
by replacing $u$ by its conormal gradient $\nabla_{A}u$ as the unknown function.
More precisely (\ref{eq:divform}) for $u$ is replaced by  \eqref{eq:firstorderODE} for 
$$
  F(t,x) =\nabla_A u(t,x)=  \begin{bmatrix} \pd_{\nu_A}u(t,x)\\ \nabla_x u(t,x) \end{bmatrix},$$ 
where $\pd_{\nu_A}u:= (A\nabla_{t,x}u)_\no$, that is the first component of $A\nabla_{t,x}u$. This is the inward  conormal derivative of $u$ for the upper half-space and the outward conormal derivative for the lower half-space.
Here we use the notation $v= \begin{bmatrix} v_{\no}\\ v_{\ta}\end{bmatrix}$ for vectors  in $\C^{1+n}$, where $v_{\no}\in \C$ is called the scalar part and $v_{\ta}\in \C^n$ the tangential part of $v$. For example, $\pd_{t}u=(\nabla_{t,x }u)_{\no}$ and $\nabla_{x} u = (\nabla_{t,x }u)_{\ta}$.

\begin{prop}  \label{prop:divformasODE}
  The pointwise transformation 
   \begin{equation}
\label{eq:hat}
A\mapsto \hat A:=  \begin{bmatrix} 1 & 0  \\ 
    c & d \end{bmatrix}\begin{bmatrix} a & b  \\ 
    0 & 1 \end{bmatrix}^{-1}=  \begin{bmatrix} a^{-1} & -a^{-1} b  \\ 
    ca^{-1} & d-ca^{-1}b \end{bmatrix}
\end{equation}
is a self-inverse bijective transformation of the set of  matrices in $\mA$.

For a pair of coefficient matrices $A= \hat B$ and $B= \hat A$, 
the pointwise map $\nabla_{t,x}u\mapsto F= \nabla_{A}u$ gives 
a one-one correspondence, with inverse $F \mapsto \nabla_{t,x}u= \begin{bmatrix} (B F)_\no \\ F_\ta  \end{bmatrix}$,
between   gradients of weak solutions $u\in  H^1_{\loc}(\R^{1+n}_{+})$ to (\ref{eq:divform})
and   solutions $F\in L^2_{\loc}(\R^{1+n}_{+}; \C^{1+n})$  of the generalized Cauchy--Riemann equations
\begin{equation}  \label{eq:firstorderODE}
  \pd_t F+  \begin{bmatrix} 0 & \divv_{x} \\ 
     -\nabla_{x} & 0 \end{bmatrix} B F=0, \quad  \curl_{x}F_{\ta}=0,
\end{equation}
where the derivatives are taken in the $\R^{1+n}_+$ distributional sense.\end{prop}

This transformation was introduced in \cite{AAM} and the proposition is proved in this generality in \cite{AA1}.  We shortly review the $L^2$ theory in \cite{AA1}.  
Denote by  $\dirac $  the self-adjoint operator on $\mH= L^2(\R^n; \C^{1+n})$ defined by $$
  \dirac:= 
    \begin{bmatrix} 0 & \divv_{x} \\ 
     -\nabla_{x} & 0 \end{bmatrix},\qquad  \dom(\dirac) = \begin{bmatrix}  \dom(\nabla)\\ \dom (\divv)\end{bmatrix}. 
$$
 The closure of the range of $\dirac$ is the set of $F\in \mH$ such that $\curl_{x}F_{\ta}=0$, that is $\clos{\ran(\dirac)}=\mH^0$.  It is shown in \cite{AKMc} that the operators $\dirac B$ and  $B\dirac $ with respective domains $B^{-1}\dom(\dirac)$  and $\dom(\dirac)$ and ranges $\ran(\dirac)$ and $B\ran(\dirac)$ are bi-sectorial operators with bounded holomorphic functional calculi on the closure of their range $\mH^0$ and $B^{-1}\mH^0$. Observe the similarity relation
 \begin{equation}
\label{eq:similar}
B(\dirac B)= (B\dirac) B \quad \mathrm{on}\ \dom(\dirac B)
\end{equation}
that allows to transfer functional properties between $\dirac B$ and $B\dirac$. 
In particular, if $\sgn(z)=1$ for $\re z>0$ and $-1$ for $\re z<0$, the operators $\sgn(\dirac B)$ and $\sgn(B\dirac)$ are well-defined bounded involutions on $\mH^0$ and $B^{-1}\mH^0$ respectively. 
 One defines the spectral spaces $\mH^{0, \pm}_{\dirac B}= \nul(\sgn(\dirac B)\mp I)$ and 
 $\mH^{0, \pm}_{B\dirac }= \nul(\sgn(B\dirac )\mp I)$. They topologically split $\mH^0$ and $B^{-1}\mH^0$ respectively. The restriction of $\dirac B$  to the invariant space $ \mH^{0, +}_{\dirac B}$ is sectorial of type less than $\pi/2$, hence it generates an analytic semi-group $e^{-t\dirac B}$, $t\ge 0$, on it. Similarly, the restriction of $B\dirac $  to the invariant space $ \mH^{0, +}_{B\dirac }$ is sectorial of type less than $\pi/2$, hence it generates an  analytic semi-group $e^{-tB\dirac }$, $t\ge 0$, on $ \mH^{0, +}_{B\dirac }$.

 \begin{thm} \label{thm:AAM-AA1} Let $u\in H^1_{loc}(\R^{1+n}_{+})$. Then,
 \begin{enumerate}
  \item  $u$ is  a weak solution of $Lu=0$ with 
  $ \|\wt N(\nabla u)\|_{2}<\infty$  if and only if there exists $F_{0} \in \mH^{0,+}_{\dirac B}$ such that  
$\nabla_{A}u=e^{-t \dirac B} F_{0}$.  Moreover, $F_{0}$ is unique  and $\|F_{0}\|_{2}\approx 
 \|\wt N(\nabla u)\|_{2}$.
  \item   $u$ is a weak solution of $Lu=0$ with $\iint_{\R^{1+n}_{+}}t|\nabla_{t,x}u|^2 dt dx<\infty$ if and only if there exists $\wt F_{0} \in \mH^{0,+}_{B\dirac }$  such that  $\nabla_{A}u=\dirac e^{-t B \dirac}\wt F_{0}$. Moreover, $\wt F_{0}$ is unique,  $\|\wt F_{0}\|_{2}\approx (\iint_{\R^{1+n}_{+}}t|\nabla_{t,x}u|^2 dt dx)^{1/2}$ and $u$ is given by $u=-(e^{-t B \dirac}\wt F_{0})_{\no}+c$ for some constant $c\in \C$. 
   
\end{enumerate} 
 
 \end{thm}

The if part was obtained in \cite{AAM} and the only if part in \cite[Theorems 8.2 and 9.2]{AA1}.  
Here $ \wt N(g)$ is the  Kenig-Pipher modified non-tangential function $ \|\wt N(\nabla u)\|_{2}$, 
where  $$
  \wt N(g)(x):= \sup_{t>0}  t^{-(1+n)/2} \|f\|_{L_2(W(t,x))}, \qquad x\in \R^n,
$$
with $W(t,x):= (c_0^{-1}t,c_0t)\times B(x;c_1t)$, for some fixed constants $c_0>1$, $c_1>0$. 

\begin{rem} Although we do not need that, the same proof  shows when coefficients are $t$-independent  that for the equivalence of (1) to hold one could replace $ \|\wt N(\nabla u)\|_{2}$ by 
the weaker condition $\sup_{t>0} (\frac 1 t \int_{t}^{2t}\|\nabla_{s,x}u\|_{2}^2\,  ds)^{1/2}$ or  by  $\sup_{t>0 } \|\nabla_{t,x}u\|_{2}$ or even by  the square function $(\iint_{\R^{1+n}_{+}}t|\pd_{t}\nabla_{t,x}u|^2 dt dx)^{1/2} $, so that in the end all these quantities are \textit{a priori} equivalent for weak solutions.

\end{rem}

In (2), the function $\wt F_{0}$ is formally built as $\dirac\wt F_{0}=\nabla_{A}u_{|t=0}$  (which belongs to  an adapted  Sobolev space of order -1: we shall make this more precise).  

The spectral spaces with negative signs correspond to estimates for solutions to $Lu=0$ in the lower-half space, and the similar statement holds using the semigroups $e^{t\dirac B}$ and $e^{tB\dirac}$.  

Our aim is to extend this formalism to Sobolev spaces. However, a difficulty is that $\ran(B\dirac)$ is really an $L^2$ object as it depends on $B$. We shall modify the setup  to prepare  this extension. Recall that $\mH^0=\clos{\ran(\dirac)}$ is a closed subspace of $\mH=L^2(\R^n; \C^{1+n})$. 
Let $S= \dirac_{|\mH^0}$ with domain $\dom(\dirac)\cap \mH^0$. Then $S$ is a  one-one, self-adjoint operator. Define $\Pi$ as the orthogonal projection from $\mH$ onto $\mH^0$. It is the identity if $n=1$ but not otherwise. Let $\bet$ be the operator on $\mH^0$ defined by
$\bet u =\Pi Bu= \Pi B\Pi u, u\in \mH^0$.  As $B$ is a strictly accretive  operator on $\mH^0$ (for equations this is true on $\mH$ but only on $\mH^0$ for systems, that is,  when $m>1$), the restriction of   $\Pi$ on $B\mH^0$ is an isomorphism onto $ \mH^0$  and  $\bet$ is a strictly accretive operator on $\mH^0$. 

Define 
    $$
T: \mH^0 \to \mH^0, \quad T= \bet S= \Pi  B\dirac_{|\mH^0} \ \mathrm{with}\  \dom(T)=\dom(S)
$$
and 
$$
\uT: \mH^0 \to \mH^0, \quad \uT= S \bet = \dirac \Pi B_{|\mH^0}= \dirac B_{|\mH^0}\ \ \mathrm{with}\  \dom(\uT)=\bet^{-1}\dom(S).
$$

\begin{thm} \label{thm:TuTi}
 \begin{enumerate}
 
  \item  $\uT$ is a one-one bi-sectorial operator  with bounded holomorphic functional calculus on $\mH^0$. A function $u\in H^1_{loc}(\R^{1+n}_{+})$ is  a weak solution of $Lu=0$ with  $ \|\wt N(\nabla u)\|_{2}<\infty$  if and only if there exists $H_{0} \in \mH^{0, +}_{\uT}$ such that  
   $\nabla_{A}u=e^{-t\uT} H_{0}$.  Moreover, $H_{0}$ is unique  and $\|H_{0}\|_{2}\approx
     \|\wt N(\nabla u)\|_{2}$.
  %   \sup_{t>0} (\frac 1 t \int_{t}^{2t}\|\nabla_{t,x}u\|_{2}^2\,  ds)^{1/2}$.
  \item  $T$ is a one-one bi-sectorial operator  with bounded holomorphic functional calculus on $\mH^0$. 
  A function $u\in H^1_{loc}(\R^{1+n}_{+})$ is a weak solution of $Lu=0$ with $\iint_{\R^{1+n}_{+}}t|\nabla_{t,x}u|^2 dt dx<\infty$ if and only if there exists $ \wt H_{0} \in \mH^{0, +}_{T}$ such that $\nabla_{A}u=S e^{-tT}\wt H_{0}$. Moreover, $\wt H_{0}$ is unique, $\|\wt H_{0}\|_{2}\sim (\iint_{\R^{1+n}_{+}}t|\nabla_{t,x}u|^2 dt dx)^{1/2}$ and $u=-(e^{-tT}\wt H_{0})_{\no}+c$ for some constant $c$. 
   
\end{enumerate} 
 
 \end{thm}

 \begin{proof} We begin with the first point. As  $\uT$ is  the restriction of $\dirac B$ to $\mH^0$, it is a one-one bi-sectorial operator with bounded holomorphic functional calculus. In particular, the spectral spaces $\mH^{0,\pm}_{\uT}$ coincide with $\mH^{0,\pm}_{\dirac B}$. The function $H_{0}$ is nothing but $F_{0}$ in Theorem \ref{thm:AAM-AA1}, (1).  This proves the first point. 
 
 Let us turn to the second point. That $T$ is one-one, bi-sectorial  with bounded holomorphic functional calculus on $\mH^0$  follows from the similarity relation $\bet\uT = T\bet$.
 To obtain the new equations for $u$ is not as direct. 
Denote by  $U: B\mH^0 \to \mH^0$  the restriction of $\Pi$  on $B\mH^0$. For $G\in \mH^0$, we have $\Pi(U^{-1} G-G)=0$ so $U^{-1}G \in \dom(\dirac) $ if and only if $G \in \dom (\dirac)$. A calculation shows that $TG= \Pi B \dirac G= UB\dirac U^{-1}G$. So $T$ is also similar to the restriction of $B\dirac $ on $B\mH^0$. 
Next,   the equation $\wt F = e^{-tB\dirac} \wt F_{0}$ with $\wt F_{0}\in \mH^{0, +}_{B\dirac}$   is equivalent  to $U \wt F= e^{-tT} U \wt F_{0}$ with $U \wt F_{0}\in \mH^{0, +}_{T}$. Hence,   $$Se^{-tT} U \wt F_{0}= \dirac U e^{-tB\dirac} \wt F_{0}=  \dirac e^{-tB\dirac} \wt F_{0}$$ and $$(e^{-tT}U \wt F_{0})_{\no}= (Ue^{-tB\dirac}\wt F_{0})_{\no}=(e^{-tB\dirac}\wt F_{0})_{\no} $$
as $U$ leaves the scalar component invariant. This concludes the proof with $\wt H_{0}=U\wt F_{0}= \Pi \wt F_{0}$, where $\wt F_{0}$ is the function specified in Theorem \ref{thm:AAM-AA1}, (2).
\end{proof}

The point is that both operators $T$ and $\uT$ act on the same space $\mH^0$ which is independent  of the action of $B$. Thus, we use the relation between quadratic estimates and interpolation developed in \cite[Section 8]{AMcN} for bi-sectorial operators and we strongly suggest the reader to have this reference handy from now on.

For $s\in \R$ define $\dot \mH^s$ as the completion of  $\mH^0$ for the quadratic norm
$$
\|F\|_{S,s}= \left\{ \int_{0}^\infty t^{-2s}\|\psi_{t}(S)F\|_{2}^2\, \frac{dt}{t}\right\}^{1/2} 
$$
where $\psi$ is a suitable holomorphic function on bi-sectors, for example $\psi(z)=z^k e^{-z\, \sgn(z)}$, $\re z\ne 0$ and $\N \ni k>\max (s,0)$. Remark that from the  spectral theorem $\dot \mH^0=\mH^0$ and  it can be checked that $\|F\|_{S,s}=c_{\psi, s}\||S|^sF\|_{2}$ where $|S| = (S^2)^{1/2}$. Note that $S$ extends to an isomorphism between $\dot \mH^s$ and $\dot \mH^{s-1}$. Classically, the intersection of $\dot \mH^s$ for $s$ in a bounded interval is dense in each of them. 

We define similarly for $\dot \mH^s_{T}$ and $\dot \mH^s_{\uT}$ replacing $S$ by $T$ and $\uT$.  The quadratic norms are equivalent under changes of suitable  $\psi$ that are non degenerate on both components of the bi-sectors. 

Note that $|S|$  preserves the normal and tangential components so we    can write $\dot \mH^s= \dot \mH^s_{\no} \oplus \dot \mH^s_{\ta}$ which agrees with our earlier notation when $s=0$ and $s=-1/2$. 
If $n=1$, we have $S=\dirac= \begin{bmatrix} 0 & d_{x} \\ 
     -d_{x} & 0 \end{bmatrix}$ so the quadratic norm   defines the 
     usual homogeneous Sobolev space $\dot H^{s}(\R; \C^{2})$ which is  also the completion of 
     $\dom( |S|^s)$ for the homogeneous norm $\||S|^sF\|_{2}$. It follows that $\dot \mH^s_{\no}= \dot \mH^s_{\ta}= \dot H^{s}(\R)$. 

For $ n\ge 2 $, recall that    $\mR=\nabla (-\Delta)^{-1/2}$ denotes the array of Riesz transforms  and let  $\mR^*=-  (-\Delta)^{-1/2}\divv$  be  its adjoint. 
The operator 
$$
V= \begin{bmatrix} I &0 \\ 
     0 & -\mR \end{bmatrix}: L^2(\R^n; \C^2) \to \mH^0
     $$
     is an isometry with inverse
     $$
    V^{-1}=V^*= \begin{bmatrix} I &0 \\ 
     0 & -\mR^* \end{bmatrix}:  \mH^0 \to L^2(\R^n; \C^2)
     $$
     and $VV^*=\Pi$.   A calculation using $\nabla= \mR (-\Delta)^{1/2}$ shows that 
     $$
     V^{-1}S V= \begin{bmatrix} 0 &(-\Delta)^{1/2} \\ 
     (-\Delta)^{1/2} & 0  \end{bmatrix}.
     $$
Thus, $\dot \mH^s_{\no} = \dot H^{s}(\R^n)$ and $\dot \mH^s_{\ta}=\mR \dot H^{s}(\R^n) $ where $\dot H^{s}(\R^n)$ is  the usual homogeneous Sobolev space with semi-norm $\|(-\Delta)^{s/2}f\|_{2}$.

\begin{prop} \label{prop:sobolev} 
\begin{enumerate}
  \item For $s\in \R$, for all bounded holomorphic functions $b$ in appropriate bi-sectors,  $b(T)$ extends to a bounded operator on $\dot \mH^s_{T}$.   In particular, this holds for $\sgn(T)$ which is a bounded self-inverse operator on $\dot\mH^s_{T}$, and $T$ and $|T|= \sgn(T) T$ extend to isomorphisms between $\dot \mH^s_{T}$ and $\dot \mH^{s-1}_{T}$.  The operator $|T|$ extends to a sectorial operator on $\dot \mH^s_{T}$.
\item   $\dot \mH^s$ topologically splits as the sum of the two spectral closed subspaces  $\dot\mH^{s,+}_{T}=\nul(\sgn (T)-I)$ and $\dot\mH^{s,-}_{T}=\nul(\sgn(T) +I)$. 
  \item The same two items hold with $T$ replaced by  $\uT$.
  \item For $0\le s\le 1$, $\dot \mH^s_{T}= \dot \mH^{s}$  and for $-1\le  s\le 0$, 
$\dot \mH^s_{\uT}= \dot \mH^{s}$ with equivalence of norms. 
  \item  Furthermore, for $-1\le s<0$, we have for  
  $\| F\|_{S,s} \approx \left\{ \int_{0}^\infty t^{-2s}\|e^{-t|\uT|}F\|_{2}^2\, \frac{dt}{t}\right\}^{1/2}$.
  % with $|\uT|=\sgn(\uT) \uT. $ 
  \end{enumerate}

\end{prop} 

\begin{proof}  The first four items are contained in \cite[Theorem 8.3]{AMcN} and the previous sections therein with the exception of the cases $s=1$ and $s=-1$ of (4). 

For $s=1$, observe that $\|F\|_{T,1}=c\||T|F\|$. Thus, the bounded holomorphic functional calculus of $T$ on $\mH^0$ implies that $\||T|F\|\approx\|TF\|$ \cite[Theorem 8.2]{AMcN}. Finally  $\|TF\|\approx \|SF\|\approx \|F\|_{S,1}$. 

For $s=-1$, we observe the intertwining relation $\uT S= ST$ (as unbounded operators) so the functional calculi also intertwine. Hence, $\psi_{t}(\uT)SF = S\psi_{t}(T)F= t^{-1}\bet^{-1}\tilde \psi_{t}(T)F$ for $F \in \dom(T)=\dom(S)$ with $\tilde \psi(z)=z\psi(z)$. This implies that $\|SF\|_{\uT, -1} \approx \|F\|_{T,0}\approx \|F\|_{S,0}\approx \|SF\|_{S,-1}$ and the result follows from the fact that the functions $SF$, $F \in \dom(T)=\dom(S)$,  form a dense subset of  $\dot \mH^{-1}$.  

To prove (5), we proceed as above and  a calculation with $\psi(z)= ze^{-z\, \sgn(z)}$ shows that 
$ e^{-t|\uT|} SF= Se^{-t|T|}F= t^{-1}\bet^{-1}\psi_{t}(T)F$  for $F\in \dom(T)$. It follows easily for appropriate $F$ and using $0 \le s+1<1$ that 
$$\left\{ \int_{0}^\infty t^{-2s}\|e^{-t|\uT|}F\|_{2}^2\, \frac{dt}{t}\right\}^{1/2} \approx \|S^{-1}F\|_{T,s+1}\approx  \|S^{-1}F\|_{S,s+1}\approx \|F\|_{S,s}.
$$
\end{proof}

We can give another useful characterization of solutions with square function estimates in terms of the negative Sobolev space.
\begin{cor}\label{cor:dir} Let $u\in H^1_{loc}(\R^{1+n}_{+})$.
Then $u$ is a weak solution of $Lu=0$ with $\iint_{\R^{1+n}_{+}}t|\nabla_{t,x}u|^2 dt dx<\infty$ if and only if there exists $ H_{0} \in \dot \mH^{-1, +}_{\uT}$  such that  $\nabla_{A}u=e^{-t\uT}   H_{0}$. Moreover, $H_{0}$ is unique and $\|H_{0}\|_{\uT, -1}\approx (\iint_{\R^{1+n}_{+}}t|\nabla_{t,x}u|^2 dt dx)^{1/2}$. 
\end{cor}

\begin{proof} This is a reformulation of the previous results. By Theorem \ref{thm:TuTi}, (2), given $u$, we have $\nabla_{A}u=Se^{-tT}\wt H_{0}= e^{-t\uT}  S\wt H_{0}$, where $e^{-t\uT} $ is now the semi-group extended to $\dot \mH^{-1, +}_{\uT}$. Setting $H_{0}=S \wt H_{0}$, we have $H_{0}\in   \dot\mH^{-1, +}_{\uT}$ as $\wt H_{0}\in  \mH^{0,+}_{T}$ and $\|H_{0}\|_{\uT, -1} \approx \|H_{0}\|_{S, -1} \approx\|\wt H_{0}\|_{2}$ by Proposition \ref{prop:sobolev}, (4). 

Conversely, let $H_{0}\in \dot\mH^{-1, +}_{\uT}$ and set $\wt H_{0}=S^{-1}H_{0} \in \dot\mH^{0, +}_{T}$. Then $e^{-t\uT}   H_{0}= e^{-t\uT}  S\wt H_{0} = Se^{-tT}  \wt H_{0}$ is the conormal gradient of a solution $u$ by Theorem \ref{thm:TuTi},~(2).
\end{proof}

Now, we prove a representation for energy solutions

\begin{prop}\label{cor:energy} Let $u\in H^1_{loc}(\R^{1+n}_{+})$.
Then $u$ is a weak solution of $Lu=0$ with $\iint_{\R^{1+n}_{+}}|\nabla_{t,x}u|^2 dt dx<\infty$ (i.e., $u$ is an energy solution) if and only if there exists $ H_{0} \in \dot \mH^{-1/2, +}_{\uT}$  such that  $\nabla_{A}u=e^{-t\uT}   H_{0}$. Moreover, $H_{0}$ is unique and $\|H_{0}\|_{\uT, -1/2}\approx (\iint_{\R^{1+n}_{+}}|\nabla_{t,x}u|^2 dt dx)^{1/2}$. 
\end{prop}

\begin{proof}  Let us prove the converse first. If $H_{0}\in \dot\mH^{-1/2, +}_{\uT}$, then $H_{\ep}=e^{-\ep\uT}H_{0} \in \dot\mH^{-1/2,+}_{\uT} \cap \dot\mH^{1/2,+}_{\uT} \subset \dot\mH^{0,+}_{\uT}$ for any $\ep>0$. By Theorem \ref{thm:TuTi}, 
$e^{-t\uT} H_{\ep}$ is the conormal gradient of some solution $u_{\ep}$. As $H_{\ep}$ converges  to $H_{0}$ in $\dot\mH^{-1/2, +}_{\uT}$, it is easy to conclude that $u_{\ep}$ converges to a solution $u$ is the energy space and  
$ (\iint_{\R^{1+n}_{+}}|\nabla_{t,x}u|^2 dt dx)^{1/2} \approx \|H_{0}\|_{\uT, -1/2} $. 

%We show  the converse. The first one is specific to energy solutions and the second one applies \textit{mutatis mutandis} to more general solutions, namely  those with $\iint_{\R^{1+n}_{+}}t^{-2s-1}|\nabla_{t,x}u|^2 dt dx<\infty$ for $-1<s<0$ . 
Assume now that $u$ is a weak solution of $Lu=0$ with $\iint_{\R^{1+n}_{+}}|\nabla_{t,x}u|^2 dt dx<\infty$.  Set $H_{0}= \nabla_{A}u_{|t=0} \in  \dot \mH^{-1/2}$. Decomposing $H_{0}=H_{0}^+
+ H_{0}^-$ according to the spectral decomposition $\dot \mH^{-1/2}= \dot \mH^{-1/2, +}_{\uT} \oplus \dot \mH^{-1/2, -}_{\uT}$ from Proposition \ref{prop:sobolev} (2) and (4), and using the implication just proved,  $ H_{0}^+$ is the trace of the conormal gradient of an energy  solution $u^+$ in the upper half-space  and, by the same result in the lower half-space,  $ H_{0}^-$ is the trace of the conormal gradient of an energy solution $u^-$ in the lower half-space. It is then easy to see that the function $v$, defined by $v= u-u^+$ on the upper half-space and $v=u^-$ in the lower half-space, is an energy solution of $Lv=0$ in $\R^{1+n}$. By Lemma \ref{lem:uniqueness}, $v=0$ (modulo constants) so that $u=u^+$ and $H_{0}=H_{0}^+$. Thus,
$\nabla_{A}u= \nabla_{A}u^+= e^{-t\uT}H_{0}^+$ from the first part of the proof.
\end{proof}

\begin{cor}\label{cor:spectral} The elements of $\dot \mH^{-1/2, +}_{\uT}$ are  $\begin{bmatrix}
      f \\ \Gamma^A_{ND}f \end{bmatrix}$, $f\in \dot \mH^{-1/2}_{\no}$. They can also be written $\begin{bmatrix} \Gamma^A_{DN} g \\ g \end{bmatrix}$, $g\in   \dot \mH^{-1/2}_{\ta}$. 
\end{cor}

\begin{proof} It suffices to check the first representation by Theorem \ref{thm:inverse}. By the previous result, $H_{0}\in \dot \mH^{-1/2, +}_{\uT}$ if and only if $H_{0}= \nabla_{A}u_{|t=0}$ for some solution $u$ in the energy space. So if $f$ is the conormal derivative $\pd_{A}u_{|t=0}$ in $\dot \mH^{-1/2}_{\no}$ then the tangential gradient at $t=0$ is $\Gamma^A_{ND}f$ by \eqref{eq:DN}. 

Conversely, if $f\in \dot \mH^{-1/2}_{\no}$, then we let $u$ be the energy solution to $Lu=0$  with Neumann datum $f$. We know  $\Gamma^A_{ND}f= \nabla_{x}u_{|t=0} $ and that $\nabla_{A}u=e^{-t\uT}H_{0}$ for some $H_{0}\in \dot \mH^{-1/2, +}_{\uT}$. It follows that $\begin{bmatrix}
      f \\ \Gamma^A_{ND}f \end{bmatrix}= H_{0} \in \dot \mH^{-1/2, +}_{\uT}$. 
\end{proof}

\begin{rem}
It is interesting to compare Theorem \ref{thm:TuTi}, Corollary \ref{cor:dir} and Proposition \ref{cor:energy}. The  use of $\uT$ allows the same representation for conormal gradients of solutions and the only difference is the space to which the trace belongs. Solving Neumann, Dirichlet and/or regularity  problems amount to proving boundedness of Neumann to Dirichlet map and/or the Dirichlet to Neumann map with different topologies on the boundary. 
\end{rem}

\begin{rem}
In \cite[section 5]{AAM}, some \textit{a posteriori} identification of the solutions coming from the $DB$ formalism in $L^2$ is made with energy solutions, upon some further assumptions. But note that  energy solutions require a different trace space. Here,  the extension of this formalism to Sobolev spaces allows to prove \textit{a priori } representation  (Proposition \ref{cor:energy}) for energy solutions. 
\end{rem}

\begin{rem} We note that Proposition \ref{cor:energy} was stated without proof for  systems on the ball in \cite{AR}, even with some  radially dependent coefficients. 
While this article was in its final stage, we learned that this kind of representations of solutions with data in Sobolev spaces was pursued in more generality by Andreas Ros\'en \cite{R}, and has some overlap with ours. 
\end{rem}

\section{The operator theoretic lemma}

The main ingredient in our proof is the next lemma which will provide a factorisation of the 
boundary maps with simpler operators to analyze. The following lemma is essentially taken from  \cite[Section 6]{AGHMc}. 

\begin{lem}\label{lem:key}
Let $X_{1}, X_{2}$ be  Banach spaces and let $\mZ= X_{1}\oplus X_{2}$ whose elements are written as 
$ \begin{bmatrix} u_{\no} \\ u_{\ta}\end{bmatrix} $. Let $S= \begin{bmatrix} s_{11}& s_{12}\\ s_{21} & s_{22} \end{bmatrix} \in \mL(\mZ)$ such that
\begin{enumerate}
  \item $S^2=I_{\mZ}$.
  \item  There exists $c\in (0,1)$ such that for all $u\in N(S\pm I)$, one has
  $$
  c^{-1} \|u_{\no}\|_{X_{1}}\le \|u_{\ta}\|_{X_{2}}\le c\|u_{\no}\|_{X_{1}}.
  $$ 
\end{enumerate}
Then, 
$s_{12}, s_{21}, s_{11}\pm I, s_{22}\pm I$ are one-one with closed range in the respective $\mL(X_{i}, X_{j})$. 
\end{lem}

\begin{proof} Let $P^\pm= \frac 12(I\pm S)$ be the pair of projectors on $\mZ$ associated to $S$ by (1) and $Q_{\pm}$ the pair of projectors  on $X_{1}\oplus \{0\}$ and $\{0\}\oplus X_{2}$.  The assumption (2) implies
$$
\|Q_{+}P^\pm u \| \approx \|Q_{-}P^\pm u \| \approx \|P^\pm u \|,  \quad u \in \mZ.
 $$
We infer that 
$$
\|P^+Q_{\pm}u\| \approx \|P^-Q_{\pm}u\| \approx \|Q_{\pm}u\|.
$$

For example,  using $Q_{-}Q_{+}=0$, then
$
 \|Q_{+}u\| \le \|P^+Q_{+}u\| + \|P^-Q_{+}u\|  \lesssim  \|P^+Q_{+}u\| + \|Q_{-}P^-Q_{+}u\| \lesssim   \|P^+Q_{+}u\| + \|Q_{-}P^+Q_{+}u\| \lesssim  \|P^+Q_{+}u\| \lesssim \|Q_{+}u\|.$
 
 Going further, one easily sees that 
$$
\|Q_{-}P^+Q_{\pm}u\| \approx \|Q_{+}P^+Q_{\pm}u\| |\approx \|Q_{\pm}u\| \approx  \|Q_{-}P^-Q_{\pm}u\|\approx  \|Q_{+}P^-Q_{\pm}u\|.
$$
 
 Thus, let $\phi \in X_{2}$ and  $u= \begin{bmatrix} 0\\ \phi\end{bmatrix}  = Q_{-}u$. One checks that 
 $$
  \begin{bmatrix} s_{12}\phi\\ 0\end{bmatrix}  = Q_{+}Su= Q_{+}(2P^+-I)u= 2Q_{+}P^+Q_{-}u$$
  and 
  $$
   \begin{bmatrix} 0\\ (I  \pm s_{22})\phi\end{bmatrix}= 2Q_{-}P^\pm u = 2Q_{-}P^+Q_{-}u,$$
   hence
   $$
   \|s_{12}\phi\|_{X_{1}}\approx \|(I \pm s_{22})\phi\|_{X_{2}} \approx \|\phi\|_{X_{2}}.
   $$
   Similarly, with $\phi \in X_{1}$ and  $u= \begin{bmatrix} \phi\\ 0\end{bmatrix}  = Q_{+}u$, one obtains 
   $$
   \|s_{21}\phi\|_{X_{2}}\approx \|(I \pm s_{11})\phi\|_{X_{1}} \approx \|\phi\|_{X_{1}}.
   $$

 \end{proof}
 
  \begin{lem}\label{lem:invertible} 
\begin{enumerate}
  \item The operator $\sgn(\uT)$ satisfies the hypotheses of the above lemma on 
  $\dot \mH^{-1/2}_{\uT}= \dot\mH^{-1/2}=\dot\mH^{-1/2}_{\no} \oplus \dot\mH^{-1/2}_{\ta}$. Moreover, the operators $s_{12}(\uT)$, $s_{21}(\uT)$, $s_{11}(\uT)\pm I$, $s_{22}(\uT)\pm I$ are invertible. 
  \item The same conclusion holds replacing $\uT$  by  $T$  on 
  $\dot \mH^{1/2}_{T}= \dot\mH^{1/2}=\dot\mH^{1/2}_{\no} \oplus \dot\mH^{1/2}_{\ta}$.
\end{enumerate}  \end{lem}

\begin{proof} We prove (1). Recall first that $\dot \mH^{-1/2}_{\uT}= \dot\mH^{-1/2}=\dot\mH^{-1/2}_{\no} \oplus \dot\mH^{-1/2}_{\ta}$ follows from Proposition \ref{prop:sobolev}, (2) and (4).  Next, equality $\sgn(\uT)^2=I$  on $\dot \mH^{-1/2}_{\uT}$  holds by the bounded holomorphic functional calculus and $\sgn^2(z)=1$ on $\C\setminus \R$. Recall that $\nul(\sgn(\uT)- I)=\mH^{-1/2,+}_{\uT}$ and  by Corollary \ref{cor:spectral} and setting $\Gamma^+(\uT)=\Gamma_{ND}^A : \dot\mH^{-1/2}_{\no} \to \dot\mH^{-1/2}_{\ta}$, one has
 for any $u\in \dot\mH^{-1/2}_{\uT}$
 $$
 u_{\ta}= \Gamma^{+}(\uT) u_{\no} \Longleftrightarrow  \begin{bmatrix} u_{\no} \\ u_{\ta}\end{bmatrix} \in \nul(\sgn(\uT)- I)=\mH^{-1/2,+}_{\uT}.
 $$
 Similarly,   the Neumann to Dirichlet map for the lower half-space yields the corresponding operator $\Gamma^-(\uT)$ characterized by 
 for any $u\in \dot\mH^{-1/2}_{\uT}$
 $$
 u_{\ta}= \Gamma^{-}(\uT) u_{\no} \Longleftrightarrow  \begin{bmatrix} u_{\no} \\ u_{\ta}\end{bmatrix} \in \nul(\sgn(\uT)+ I)=\mH^{-1/2,-}_{\uT}.
 $$
Details are left to the reader. Thus, the second assumption of Lemma \ref{lem:key} is granted and we have lower bounds for all six operators in the statement. Thus, when $A=Id$, $\uT= S$ and so $\sgn(\uT)=\sgn(S)=\begin{bmatrix} 0 & H\\ H^* &0\end{bmatrix}$ if $n=1$ where $H$ is the Hilbert transform and $\sgn(\uT)=\sgn(S)=\begin{bmatrix} 0 & -\mR^*\\ -\mR &0\end{bmatrix}$  on $\mH^0$ if $n\ge2$. Thus, the six operators are invertible in this case and in particular onto. The ontoness for general $A$ follows from the method of continuity. This proves (1). 

To prove (2), we observe that the intertwining relation $S\sgn(T)=\sgn(\uT) S$ extends to $\dot \mH^{1/2}_{T}= \dot\mH^{1/2}$ using that $S$ is an isomorphism from $ \dot\mH^{1/2}$ onto 
 $\dot\mH^{-1/2}$. Thus, the conclusion follows straightforwardly from (1). 
\end{proof}

 \begin{lem} \label{lem:gamma}  In  $\mL(\dot \mH^{-1/2}_{\no}, \dot \mH^{-1/2}_{\ta})$, 
 $$
 \Gamma^A_{ND}= s_{12}(\uT)^{-1} (I-s_{11}(\uT))= (I-s_{22}(\uT))^{-1} s_{21}(\uT) .
 $$
\end{lem}

\begin{proof}
We have seen that the operators  $\Gamma^\pm(\uT)$ are bounded and invertible. To obtain the formulas, one uses their operator defining relations
$$
\begin{bmatrix} s_{11}(\uT) & s_{12}(\uT) \\ s_{21}(\uT)  & s_{22}(\uT)  \end{bmatrix}  \begin{bmatrix} I  \\ \Gamma^\pm(\uT) 
 \end{bmatrix}=  \pm \begin{bmatrix} I  \\ \Gamma^\pm(\uT)
 \end{bmatrix}, 
 $$
 and then solve for $\Gamma^\pm(\uT) $ using the invertibilities of the operators in Lemma \ref{lem:invertible}. We conclude using $ \Gamma^A_{ND}=\Gamma^+(\uT)$. (As mentioned, the operator $\Gamma^-(\uT)$ is the Neumann to Dirichlet operator for the lower half-space and has similar representations $\Gamma^-(\uT)=-s_{12}(\uT)^{-1} (I+s_{11}(\uT))= -(I+s_{22}(\uT))^{-1}s_{21}(\uT)$.)
 \end{proof}

\section{block triangular matrices}

So far, everything holds for arbitrary ($t$ independent) coefficients and can be used in full generality.
Assume that $A$ is block lower-triangular: $A(x)=\begin{bmatrix}  a(x) & 0\\ c(x) &d(x) \end{bmatrix}$. In the PDE language for $L$, this means that the conormal derivative is proportional to the $t$ derivative:  $\pd_{\nu_{A}}= a\pd_{t}$. 
 If $B=\hat A$, as in the second section, then this is equivalent to  $B$ being   block lower-triangular as well.

Let us write $B(x)=\begin{bmatrix}  a'(x) & 0\\ c'(x) &d'(x) \end{bmatrix}$, which we identify with the operator of multiplication by $B$. Then $\bet=\Pi B\Pi$ also has the same structure $\begin{bmatrix}  \alpha & 0\\ \gamma &\delta  \end{bmatrix}$, where $\alpha,\gamma,\delta $ are the following operators: if $n=1$, they are the operators of  multiplication by $a',c',d'$ respectively. If $n\ge 2$, then $\alpha=a'$,  $\gamma=\mR\mR^*c'$ and $\delta =\mR\mR^*d'\mR\mR^*$.

\begin{lem} If the coefficients are block lower-triangular then $s_{12}(\uT)$ is invertible from    $\mH^0_{\ta}$ onto $\mH^0_{\no}$. 
\end{lem}

\begin{proof} We write the proof when $n\ge 2$. The proof when $n=1$ is similar and simpler  as we do not have to consider the Riesz transforms. By Lemma \ref{lem:invertible}, $s_{12}(\uT)$ is invertible from  $\dot\mH^{-1/2}_{\ta}$ onto  $\dot\mH^{-1/2}_{\no}$. By the characterization of  $\dot\mH^{-1/2}_{\ta}$ using $\mR$, this is equivalent to $s_{12}(\uT)\mR$ invertible from $\dot H^{-1/2}(\R^n)$ onto itself.  Remark that the same holds for
$s_{12}(T)\mR$ from $\dot H^{1/2}(\R^n)$ onto itself. Thus, looking at the upper off-diagonal coefficient in the relation $\bet \sgn(\uT)\Pi=\sgn(T) \bet$ between bounded operators for the $L^2$ topology, we find
$$\alpha s_{12}(\uT)\mR\mR^* = s_{12}(T) \delta. $$
Hence we have the equality in $\mL(L^2(\R^n))$
$$
s_{12}(\uT)\mR = a'^{-1} (s_{12}(T)\mR) (\mR^*d'\mR)
$$
and this implies that $s_{12}(\uT)\mR$ extends to an invertible operator from $(\mR^*d'\mR)^{-1}\dot H^{1/2}(\R^n)$ onto $a'^{-1}\dot H^{1/2}(\R^n)$.  Using the complex interpolation equalities
$$[a'^{-1}\dot H^{1/2}(\R^n), \dot H^{-1/2}(\R^n)]_{1/2}= L^2(\R^n)$$ and 
$$[(\mR^*d'\mR)^{-1}\dot H^{1/2}(\R^n), \dot H^{-1/2}(\R^n)]_{1/2}= L^2(\R^n),$$ we obtain the desired conclusion. 

It remains to explain the interpolation equalities. The first one was actually first observed (without proof) in \cite{McM} as a consequence of the solution of the 1-dimensional  Kato square root problem \cite{CMcM} using that $a'$ is an accretive function. See  \cite[Section 6]{McN} for some account.  It can also be seen as a consequence of the Tb theorem. See \cite[Th\'eor\`eme 9]{DJS} for an explicit statement. 

The second equality is  a consequence of the solution of the $n$-dimensional Kato square root problem  \cite{AHLMT} and the results  in \cite{AMcN} as follows. Remark that $\mR^*d'\mR$ is a strictly accretive  and bounded operator on $L^2(\R^n)$.  Combining \cite[Theorems 5.3, 7.2, 7.3]{AMcN} using that $\dot H^s(\R^n)$ is defined by the semi-norms $\|(-\Delta)^{s/2}f\|_{2}$, we have
$$[(\mR^*d'\mR)^{-1}\dot H^{1/2}(\R^n), \dot H^{-1/2}(\R^n)]_{1/2}
%= [(\mR^*d\mR)^{-1}\dot H^{1}(\R^n), \dot H^{-1}(\R^n)]_{1/2}
= \dot \mH^0_{\mT},
$$
where $\mT$ is the sectorial operator $ (-\Delta)^{1/2}(\mR^*d'\mR)$. But  $\dot \mH^0_{\mT}= L^2(\R^n)$ means that $\mT$ has a bounded holomorphic functional calculus on $L^2(\R^n)$, which is the same as $(\mR^*d'\mR)(-\Delta)^{1/2}$ has a bounded holomorphic functional calculus on $L^2(\R^n)$.  By  \cite[Theorem 10.1]{AMcN}, this is equivalent to  the claim 
$$\|L_{\ta}^{1/2}f\|_{2}\approx \|(-\Delta)^{1/2}f\|_{2},  \quad f\in \dom(L_{\ta}^{1/2}),$$ where $L_{\ta}$ is the operator $(-\Delta)^{1/2}(\mR^*d'\mR)(-\Delta)^{1/2}$. As $L_{\ta}$ is nothing but the divergence operator $-\divv d' \nabla$, the claim is proved in \cite{AHLMT}. 
\end{proof}

\begin{thm}\label{thm:neu}
The Neumann problem for operators $L$ with block lower-triangular, $t$-independent  coefficients $A$ is well-posed for $L^2$ data. 
\end{thm}

\begin{proof} From the   relation $ \Gamma^A_{ND}= s_{12}(\uT)^{-1} (I-s_{11}(\uT))$ and the previous lemma, we have that $\Gamma^A_{ND}$ is bounded from $\dot \mH^{0}_{\no}$ to 
$\dot \mH^{0}_{\ta}$. Thus, given a Neumann data $f\in L^2(\R^n)$, $\Gamma^A_{ND}f \in L^2(\R^n, \C^n)$, and from this it is easy to conclude that $H_{0}=\begin{bmatrix} f\\ \Gamma^A_{ND}f\end{bmatrix} \in \mH^{0,+}_{\uT}$ with $\|H_{0}\|_{2}\approx \|f\|_{2}$. By Theorem \ref{thm:TuTi}, (i),  we can define a weak solution $u$ of $Lu=0$ with 
 $ \|\wt N(\nabla u)\|_{2} \approx \|H_{0}\|_{2}$ by $\nabla_{A}u=e^{-t\uT} H_{0}$.
\end{proof}

\begin{rem}\label{sneiberg} Under the same assumptions, 
by interpolation, $\Gamma^A_{ND}$ is bounded from $\dot \mH^{-s}_{\no}$ to 
$\dot \mH^{-s}_{\ta}$ for $0 \le s\le 1/2$.  When $0<s\le 1/2$, given any $f\in \dot H^{-s}(\R^n)$  the above formalism furnishes a solution to $Lu=0$ with conormal derivative  $f$  and $\iint_{\R^{1+n}_{+}}t^{-1+2s}|\nabla_{t,x}u|^2 dt dx \approx \|f\|_{\dot H^{-s}(\R^n)}^2$. By Sneiberg's perturbation result \cite{sneiberg}, this can be pushed to $s<1/2+\ep$ for some $\ep>0$ because $s_{11}(\uT) \in \mL(\dot  \mH^{-s}_{\no})$ and $s_{12}(\uT) \in \mL(\dot \mH^{-s}_{\ta}, \dot \mH^{-s}_{\no})$  for all $s\in [0,1]$, the spaces are complex interpolation scales and invertibility holds at $s=1/2$. \end{rem}

The same arguments apply to block upper-triangular coefficients: this time, the lower off-diagonal block coefficient $c'$ vanishes but not the upper off-diagonal block coefficient $b'$.

\begin{thm}\label{thm:reg}
The regularity  problem for operators $L$ with  block upper-triangular $t$-independent  coefficients $A$ is well-posed for $L^2$ data. 
\end{thm}

\begin{proof} It follows from Theorem \ref{thm:inverse} and Lemma \ref{lem:gamma} that 
$\Gamma^A_{DN}= s_{21}(\uT)^{-1}(I-s_{22}(\uT))$ in $\mL(\dot \mH^{-1/2}_{\ta}, \dot \mH^{-1/2}_{\no})$ from  the invertibility of $s_{21}(\uT)$  from $\dot\mH^{-1/2}_{\no}$ onto  $\dot\mH^{-1/2}_{\ta}$. By Theorem \ref{thm:TuTi}, (i),  it is enough to show  the invertibility of $s_{21}(\uT)$  from $\dot\mH^{0}_{\no}$ onto  $\dot\mH^{0}_{\ta}$, that is the invertibility of $\mR^*s_{21}(\uT)$ on $L^2(\R^n)$. This is based on the interpolation  argument and the formula
$$
\mR^*s_{21}(\uT)a'= ( \mR^*d'\mR )(\mR^*s_{21}(T))
$$
in $\mL(L^2(\R^n))$ that can be checked as above using that the coefficient $c'$ vanishes. Details are left to the reader. 
\end{proof}

\begin{rem} Observations similar to the ones in Remark \ref{sneiberg} apply for regularity problems with data (the tangential gradient at the boundary) in $\dot \mH^{-s}_{\ta}$ for block upper-triangular coefficients.
\end{rem}

\begin{thm}\label{thm:dir}
The Dirichlet problem for operators $L$ with block lower-triangular $t$-independent  operators $A$ is well-posed for $L^2$ data. 
\end{thm}

\begin{proof} Here, well-posedness is within the class of solutions with $\iint_{\R^{1+n}_{+}}t|\nabla_{t,x}u|^2 dt dx<\infty$ and by \cite{AA1}, this implies then  that $\|\wt N(u)\|_{2}\approx  \iint_{\R^{1+n}_{+}}t|\nabla_{t,x}u|^2 dt dx \approx \|u_{|t=0}\|_{2}$. 
Well-posedness follows from the duality principle \cite[Proposition 17.6]{AR} that the regularity problem for $L$  with data in $L^2$  is well-posed in the class with modified non-tangential estimate if and only if the Dirichlet problem for  $L^*$  with $L^2$ data  is well-posed in the above class.
\end{proof}

\begin{rem} We point out that the relation between $L$ and $A$ is a correspondence only when $L$ and $A$ are imposed to be self-adjoint. Otherwise,  one $L$ may be represented by many different strictly accretive matrices $A$. As a consequence, there are as many Neumann problems for $L$ as choices of $A$ since the conormal derivative depends on $A$. However, for the regularity and Dirichlet problems, the choice of $A$ is irrelevant.  Although the solution algorithms depend on the coefficients $A$, they produce in the end the same solutions (note that the tangential part of the conormal gradient is independent of the coefficients).
 
For example, one does not modify $L$ by adding to $A$ any matrix of the form $M_{\gamma}=\begin{bmatrix}  0 & \gamma(x)^t\\ -\gamma(x) & 0\end{bmatrix}$ with $\gamma$ a $\R^n$-valued, measurable and bounded function with $\divv_{x} \gamma=0.$ First, the divergence free equation implies $-\divv_{t,x} (A+M_{\gamma}) \nabla_{t,x} = -\divv_{t,x} A \nabla_{t,x}$ in the sense of forms. Second,   the accretivity constant of  $A+M_{\gamma}$ is the lower bound of $\frac 1 2 (A +M_\gamma+A^*+M_{\gamma}^*)$, but  $M_\gamma+M_{\gamma}^*=0$ when $\gamma$ is $\R^n$-valued, so this is the same as the accretivity constant of $A$.  
This argument  also shows that $A+M_{\gamma}$ remains strictly accretive if  $\gamma$ is $\C^n$-valued provided the norm of its imaginary part is not too large. 

For systems of size $m$, this example has to be modified as follows.  $\R^n$-valued is replaced by  $\mS^n$-valued where $\mS$ is the space of real and symmetric $m\times m$ matrices. The divergence free equation is understood  coefficientwise $\sum_{j=1}^n\pd_{j} \gamma_{j}^{\alpha,\beta}=0$. 
\end{rem}

\begin{rem}
Recall that any of these well-posedness results is stable under complex perturbation in $L^\infty$ within the class of $t$-independent coefficients $A$  by \cite[Theorem 2.2]{AAM} and  \cite{AA1} (the latter reference extends  the solution spaces for the  results of the former to hold).
\end{rem}

\bibliographystyle{acm}
%GATHER{AKMcDirac.bib}  % makes sure WinEdt finds citations...
%\bibliography{tdependent}

\end{document}